\documentclass[a4paper]{article}
\usepackage{mathrsfs, amsmath, amsthm,amssymb,bm}% packages for math
\usepackage{graphicx, xcolor, tikz}% packages for figures
\usetikzlibrary{arrows.meta, intersections, calc}
\usepackage{tkz-graph, tkz-berge}% If you use the xcolor package, load that package before tkz-graph to avoid package conflicts.
\usepackage{caption, subcaption}
\usepackage{array}
\usepackage{cases}
\makeatletter
\def\th@plain{%
  \upshape %\itshape % body font
}
\makeatother

\makeatletter
\renewenvironment{proof}[1][\proofname]{\par
  \pushQED{\qed}%
  \normalfont \topsep6\p@\@plus6\p@\relax
  \trivlist
  \item[\hskip\labelsep
        \bfseries
    #1\@addpunct{.}]\ignorespaces
}{%
  \popQED\endtrivlist\@endpefalse
}
\makeatother

\newtheorem{theorem}{Theorem}[section]
\newtheorem{lemma}{Lemma}
\newtheorem{corollary}{Corollary}

\newtheorem{conjecture}{Conjecture}
\newtheorem*{conjecture*}{Conjecture}
\newtheorem{case}{Case}
\newtheorem{subcase}{Subcase}[case]

\newtheorem{claim}{Claim}

\newtheorem{problem}{Problem}

\theoremstyle{definition}

\newtheorem{remark}{Remark}

\usepackage[left=25mm,top=25mm,bottom=25mm,right=25mm]{geometry}
\usepackage[T1]{fontenc} % https://tex.stackexchange.com/questions/664/why-should-i-use-usepackaget1fontenc
\usepackage{newtxtext}
\usepackage[varvw]{newtxmath}

\usepackage{paralist, tablists}% put these packages ahead of enumitem
\usepackage[inline]{enumitem}
\usepackage[square, numbers, sort&compress]{natbib}

\usepackage[pdftex,%
  bookmarks=true,%
  bookmarksnumbered=true, %% 把章节的标号放入书签中，缺省为flase
  bookmarksopen=true, %% 打开书签树
  plainpages=false,%
  pdfpagelabels,%
  colorlinks=true, %% 彩色文字链接打开，去掉方框
  linkcolor=blue, %% 目录链接的颜色为蓝色
  citecolor=blue,%
  anchorcolor=green,
  urlcolor= blue,
  breaklinks=true,
  hyperindex=true]{hyperref}%% 不能和cite宏包同时使用

\newcounter{Hcase}
\newcounter{Hclaim}

\newcommand{\etal}{et~al.\ }
\newcommand{\ie}{i.e.,\ }

\def\int(#1){\mathrm{int}(#1)}
\def\ext(#1){\mathrm{ext}(#1)}
\def\Int(#1){\mathrm{Int}(#1)}
\def\Ext(#1){\mathrm{Ext}(#1)}
\def\ad(#1){\mathrm{ad}(#1)}
\def\mad(#1){\mathrm{mad}(#1)}
\def\la(#1){\mathrm{la}(#1)}
\newcommand{\Lfloor}{\left\lfloor}
\newcommand{\Rfloor}{\right\rfloor}

\newtheorem*{TCC}{Total Coloring Conjecture}
\begin{document}%
\title{Total coloring of 1-toroidal graphs of maximum degree at least 11 and no adjacent triangles}
\author{Tao Wang\footnote{{\tt Corresponding
author: wangtao@henu.edu.cn}}\\
{\small Institute of Applied Mathematics}\\
{\small Henan University, Kaifeng, 475004, P. R. China}}
\date{}
\maketitle
\begin{abstract}%
  A {\em total coloring} of a graph $G$ is an assignment of colors to the vertices and the edges of $G$ such that every pair of adjacent/incident elements receive distinct colors. The {\em total chromatic number} of a graph $G$, denoted by $\chiup''(G)$, is the minimum number of colors in a total coloring of $G$. The well-known Total Coloring Conjecture (TCC) says that every graph with maximum degree $\Delta$ admits a total coloring with at most $\Delta + 2$ colors. A graph is {\em $1$-toroidal} if it can be drawn in torus such that every edge crosses at most one other edge. In this paper, we investigate the total coloring of $1$-toroidal graphs, and prove that the TCC holds for the $1$-toroidal graphs with maximum degree at least~$11$ and some restrictions on the triangles. Consequently, if $G$ is a $1$-toroidal graph with maximum degree $\Delta$ at least~$11$ and without adjacent triangles, then $G$ admits a total coloring with at most $\Delta + 2$ colors.
\end{abstract}
\section{Introduction}
  All graphs considered are finite, simple and undirected unless otherwise stated. Let $G$ be a graph with vertex set $V$ and edge set $E$. We shall denote by $F(G)$ the set of faces of an embedded graph $G$. The {\em neighborhood} of a vertex $v$ in a graph $G$, denoted by $N_{G}(v)$, is the set of all the vertices adjacent to the vertex $v$, \ie $N_{G}(v) = \{\,u \in V(G) \mid uv \in E(G)\,\}$. The {\em degree} of a vertex $v$ in $G$, denoted by $\deg_{G}(v)$, is the number of edges of $G$ incident with $v$. We denote the minimum and maximum degree of vertices of $G$ by $\delta(G)$ and $\Delta(G)$, respectively. The {\em diamond graph} $K_{4}^{-}$ is the graph $K_{4}$ minus an edge. A graph property $\mathcal{P}$ is {\em deletion-closed} if $\mathcal{P}$ is closed under taking subgraphs. A graph is {\em diamond-free} if it contains no induced subgraph which is isomorphic to $K_{4}^{-}$. In an embedded graph $G$, the {\em degree} $\deg_{G}(f)$ of a face $f$ is the number of edges with which it is incident, cut edge being counted twice. A $d$-vertex, $d^{+}$-vertex and $d^{-}$-vertex is a vertex of degree $d$, at least $d$ and at most $d$, respectively. Analogously, a $d$-face, $d^{+}$-face and $d^{-}$-face is a face of degree $d$, at least $d$ and at most $d$, respectively.

  A {\em total coloring} of a graph $G$ is an assignment of colors to the vertices and the edges of $G$ such that every pair of adjacent/incident elements receive distinct colors. The {\em total chromatic number} of a graph $G$, denoted by $\chiup''(G)$, is the minimum number of colors in a total coloring of $G$. It is obvious that the total chromatic number of a graph $G$ has a trivial lower bound $\Delta(G) + 1$. For the upper bound, Behzad \cite{Behzad1965} raised the following well-known Total Coloring Conjecture (TCC):

\begin{TCC}
Every graph with maximum degree $\Delta$ admits a total coloring with at most $\Delta + 2$ colors.
\end{TCC}

The conjecture was verified in the case $\Delta =3$ by Rosenfeld \cite{MR0278995} and Vijayaditya \cite{MR0285447} independently and also by Yap \cite{MR0976059}. It was confirmed in the case $\Delta \in \{4, 5\}$ by Kostochka \cite{MR0453576, MR1425788}, in fact the proof holds for multigraphs. Regarding planar graphs, the conjecture was verified in the case $\Delta \geq 9$ by Borodin \cite{MR977440} and in the case $\Delta = 7$ by Sanders and Zhao \cite{MR1684286}; the case $\Delta = 8$ was a consequence of Vizing's theorem about planar graphs \cite{Vizing1965} and Four Color Theorem (for more details, see Jensen and Toft \cite{MR1304254}). Thus, the only remaining case for planar graphs is that of maximum degree six.

The following conjecture is equivalent to the TCC, but it is more suitable for proof by contradiction. Throughout the paper, we consider the following form of the TCC.
\begin{conjecture}\label{TCC*}%
Every graph with maximum degree at most $\Delta$ admits a total coloring with at most $\Delta + 2$ colors.
\end{conjecture}

\section{Preliminary}
  A {\em $\kappa$-deletion-minimal} graph with respect to total coloring, is a graph with maximum degree at most $\kappa - 1$ such that its total chromatic number is greater than $\kappa$, but the total chromatic number of every proper subgraph is at most $\kappa$. A $\kappa$-deletion-minimal graph $G$ with $\kappa \geq \Delta(G) + 2$ has the following structural results.
\begin{lemma}\label{DEDGE} 
If $u$ and $v$ are two adjacent vertices with $\deg_{G}(v) \leq \Lfloor \frac{\kappa - 1}{2} \Rfloor$, then $\deg_{G}(u) + \deg_{G}(v) \geq \kappa + 1$.
\end{lemma}
\begin{proof}
  Suppose, to the contrary, that $\deg_{G}(u) + \deg_{G}(v) \leq \kappa$. By the minimality of $G$, the graph $G - uv$ admits a total coloring $\phi$ with at most $\kappa$ colors. Let $\varphi$ denote the coloring obtained from $\phi$ by removing the color of $v$. Since $\deg_{G - uv}(u) + \deg_{G - uv}(v) \leq \kappa - 2$, it is easy to extend $\varphi$ to the edge $uv$ by assigning an available color. Finally, we can assign a color to $v$ such that the resulting coloring is a total coloring since $2\deg_{G}(v) \leq \kappa - 1$.
\end{proof}
\begin{lemma}\label{MinimumDegree} 
The graph $G$ is $2$-connected and $\delta(G) \geq 3$.
\end{lemma}
\begin{proof}
It is obvious that $G$ is $2$-connected. If $v$ is a vertex of degree at most two, then \autoref{DEDGE} implies that every neighbor of $v$ has degree at least $\kappa + 1 - \deg_{G}(v) \geq \Delta + 1$, which is a contradiction. Thus, we have $\delta(G) \geq 3$.
\end{proof}
\begin{lemma}\label{SumDelta3}
If $u$ and $v$ are two adjacent vertices with $\deg_{G}(v) \leq \Lfloor \frac{\kappa - 1}{2} \Rfloor$ and $\deg_{G}(u) + \deg_{G}(v) \leq \kappa + 1$, then the edge $uv$ is not contained in any triangle in $G$.
\end{lemma}
\begin{proof}%
  Suppose that $uv$ is contained in a triangle $uvw$. By the minimality of $G$, the graph $G - uv$ admits a total coloring $\phi$ with at most $\kappa$ colors. Let $\pi$ denote the coloring obtained from $\phi$ by removing the color of $v$. Let $\mathcal{U}_{\pi}(v)$ denote the set of colors which are assigned to the edges incident with $v$, and let $\mathcal{U}_{\pi}(u)$ denote the set of colors which are assigned to the vertex $u$ or the edges incident with $u$. Suppose that $\{1, \dots, \kappa\}$ is not the union of $\mathcal{U}_{\pi}(v)$ and $\mathcal{U}_{\pi}(u)$. Hence, there exists a color $\theta$ which is missed at $u$ and $v$, assign $\theta$ to $uv$ and assign a suitable color to $v$, it yields a total coloring of $G$ with at most $\kappa$ colors, which is a contradiction. Therefore, the set $\{1, \dots, \kappa\}$ is the union of $\mathcal{U}_{\pi}(v)$ and $\mathcal{U}_{\pi}(u)$; in fact, it is the disjoint union of $\mathcal{U}_{\pi}(v)$ and $\mathcal{U}_{\pi}(u)$ since $|\mathcal{U}_{\pi}(v)| + |\mathcal{U}_{\pi}(u)| = \kappa$. Note that $\pi(wv) \notin \mathcal{U}_{\pi}(u)$. From the coloring $\pi$, remove the color on $wv$ and assign the color $\pi(wv)$ to $uv$, we obtain a total coloring $\psi$ of $G - wv$ except $v$. Let $\mathcal{U}_{\psi}(w)$ denote the set of colors which are assigned to the vertex $w$ or the edges incident with $w$ with respect to $\psi$. Similarly, we can prove that $\{1, \dots, \kappa\}$ is the union (not necessarily disjoint union) of $\mathcal{U}_{\pi}(v)$ and $\mathcal{U}_{\psi}(w)$. Therefore, we have $\mathcal{U}_{\pi}(u) \subseteq \mathcal{U}_{\psi}(w) \subseteq \mathcal{U}_{\pi}(w)$. In the coloring $\pi$, there is a color $\alpha \notin \mathcal{U}_{\pi}(u) \cup \mathcal{U}_{\pi}(w)$, reassigning $\alpha$ to $uw$ and assigning $\pi(uw)$ to $uv$, and giving a suitable color to $v$, yields a total coloring of $G$ with at most $\kappa$ colors, which derives a contradiction.
\end{proof}
\begin{lemma}\label{No3InTriangle}
If $v$ is a $3$-vertex and $\kappa \geq 7$, then $N_{G}(v)$ is an independent set \cite[Lemma 3]{MR3352881}.
\end{lemma}
\begin{proof}%
Similar result has been proved in \cite[Lemma 3]{MR3352881}. Here, we can directly apply \autoref{DEDGE} and \autoref{SumDelta3} to obtain it. 
\end{proof}

\begin{lemma}\label{No4InAdjacent}
If $v$ is a $4$-vertex and $\kappa \geq 9$, then no edge incident with $v$ is contained in two triangles \cite[Lemma 4]{MR3352881}.
\end{lemma}

  A graph is {\em $1$-embeddable} in a surface $S$ if it can be drawn in $S$ such that every edge crosses at most one other edge. In particular, a graph is {\em $1$-toroidal} if it can be drawn in torus such that every edge crosses at most one other edge; a graph is {\em $1$-planar} if it can be drawn in the plane such that every edge crosses at most one other edge. The concept of $1$-planar graph was introduced by Ringel \cite{MR0187232} in 1965, while he simultaneously colors the vertices and faces of a plane graph such that any pair of adjacent/incident elements receive distinct colors. Ringel \cite{MR0187232} proved that $1$-planar graphs are $7$-colorable, and conjectured that they are $6$-colorable, this conjecture was proved by Borodin \cite{MR832128, MR1333779}.

  Obviously, planar graphs are $1$-planar graphs and $1$-planar graph is an extension of planar graph in some sense. Zhang \etal \cite{MR3352881} proved the TCC holds for $1$-planar graphs with maximum degree at least $13$. For other various colorings of $1$-planar graphs, see \cite{MR1865580, MR2965951, MR2779909, MR2876230}. From the definitions, planar graphs and $1$-planar graphs are all $1$-toroidal graphs.

\begin{figure}%
\centering
\includegraphics{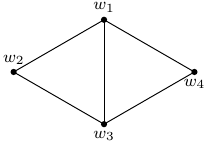}
\caption{The diamond graph $K_{4}^{-}$}
\label{Diamond}
\end{figure}
A graph $G$ has property $\mathcal{P}$, if it satisfies the following two conditions:
\begin{enumerate}[label=(\arabic*)]
\item every subgraph $K_{4}$ has at least one vertex of degree at most four;
\item every induced subgraph $K_{4}^{-}$ (see Fig~\ref{Diamond}) has $\min \{\deg_{G}(w_{1}), \deg_{G}(w_{3})\} \leq 5$ or $\min \{\deg_{G}(w_{2}), \deg_{G}(w_{4})\} \leq 3$.
\end{enumerate}

  Suppose that $K_{4}$ is a subgraph of $G - e$, thus it is also a subgraph of $G$ and it has at least one vertex of degree at most four in $G$ (also in $G - e$). Let $K_{4}^{-}$ be an induced subgraph of $G - e$ (see Fig~\ref{Diamond}). If it is also an induced subgraph of $G$, then it satisfies the condition~(2) for $G - e$. Suppose that its vertices induced a $K_{4}$ in $G$ and $e = w_{2}w_{4}$. Since this $K_{4}$ satisfies the condition~(1) for $G$, we may assume that one of its vertex $w$ has degree at most four in $G$. If $w \in \{w_{2}, w_{4}\}$, then $\deg_{G-e}(w) \leq 3$ and this $K_{4}^{-}$ satisfies the condition~(2) for $G - e$. If $w \in \{w_{1}, w_{3}\}$, then $\deg_{G-e}(w) \leq 4$ and this $K_{4}^{-}$ also satisfies the condition~(2) for $G - e$. Therefore, the property $\mathcal{P}$ is deletion-closed. 

  In this paper, we investigate the total coloring of $1$-toroidal graphs, and prove that the TCC holds for the $1$-toroidal graphs with property $\mathcal{P}$.

  Two triangles are {\em adjacent} if they have one common edge. Let $G$ be a graph drawn in a surface; if we treat all the crossing points as vertices, then we obtain an embedded graph $G^{\dagger}$, and call $G^{\dagger}$ {\em the associated graph of $G$}, call the vertices of $G$ {\em true vertices} and the crossing points {\em crossing vertices}.
\section{Total coloring}

\begin{theorem}\label{MainResult}%
Let $G$ be a $1$-toroidal graph with maximum degree at most $\Delta$, where $\Delta \geq 11$. If $G$ satisfies property $\mathcal{P}$, then $G$ admits a total coloring with at most $\Delta + 2$ colors.
\end{theorem}

Consequently, we have the following corollaries.
\begin{corollary}
Let $G$ be a diamond-free $1$-toroidal graph with maximum degree at most $\Delta$, where $\Delta \geq 11$. If every subgraph $K_{4}$ has a vertex of degree at most four, then $G$ admits a total coloring with at most $\Delta + 2$ colors.
\end{corollary}
\begin{corollary}%
Let $G$ be a $1$-toroidal graph with maximum degree at most $\Delta$, where $\Delta \geq 11$. If $G$ has no adjacent triangles, then $G$ admits a total coloring with at most $\Delta + 2$ colors.
\end{corollary}
  We prove the \autoref{MainResult} by contradiction. Let $G$ be a counterexample to the theorem with $|V| + |E|$ is minimum, and fix $\kappa = \Delta + 2$. We also assume that it has been $2$-cell $1$-embedded in the plane/torus (that is, every face of its associated graph is homeomorphic to an open disk). Since the property $\mathcal{P}$ is deletion-closed and every proper subgraph of $G$ is also a $1$-toroidal graph, it follows that $G$ is a $\kappa$-deletion-minimal graph. Let $G^{\dagger}$ be the associated graph of $G$. It is easy to see that $G^{\dagger}$ is also $2$-connected and every face boundary walk is a cycle of $G^{\dagger}$. If there exists a $4^{+}$-face $f$ with two discontinuous true vertices on the boundary walk and these two vertices have degree at most five in the current graph, then we add a line linking these two vertices in the face $f$, and call this line a {\em new edge}. We constantly add new edges one by one, and obtain an embedded graph $G^{*}$. The aim of adding new edges is to partition "big" faces into "smaller" faces such that no two discontinuous true vertices on the boundary walk has "small" degree (at most five). By the construction, every face boundary walk of $G^{*}$ is a cycle and the maximum degree of $G^{*}$ is still at most $\Delta$. Note that maybe $G^{*}$ has multiple edges, but every face is a $3^{+}$-face. If $e_{1}$ and $e_{2}$ are multiple edges, then both are new edges; otherwise, one of them is a new edge and the other is an edge of $G$, which contradicts \autoref{DEDGE}. We notice that the crossing vertices are independent in $G^{*}$.

  A vertex in $G^{*}$ is called a {\em $(d_{1}, d_{2})$-vertex}, if it has degree $d_{1}$ in $G$ and $d_{2}$ in $G^{*}$. A vertex $v$ is called {\em big} if it is a $(3, 5)$-vertex or $\deg_{G^{*}}(v) \geq 6$; otherwise, it is called a {\em small} vertex (including the crossing vertices).

By Euler's formula, we have
\begin{equation}\label{*}%
\sum_{v \in V(G^{*})}(\deg_{G^{*}}(v)-6) + \sum_{f \in F(G^{*})}(2\deg_{G^{*}}(f) - 6) = -12\textrm{ or }0.
\end{equation}

  We will use the discharging method to complete the proof. The initial charge of every vertex $v$ is $\deg_{G^{*}}(v)-6$, and the initial charge of every face $f$ is $2\deg_{G^{*}}(f) - 6$. It follows that the sum of charge of vertices and faces is at most zero by \eqref{*}. We then transfer some charge from the $4^{+}$-faces and some big vertices to small vertices, such that the final charge of every small vertex becomes nonnegative and the final charge of every big vertex and face remains nonnegative, but there is at least one element's final charge is positive, and thus the sum of the final charge of vertices and faces is positive, which derives a contradiction.
\begin{claim}\label{No4Clique}
There is no four vertices induced a $K_{4}$ in $G$.
\end{claim}
\begin{proof}
  Suppose that $\{v_{1}, v_{2}, v_{3}, v_{4}\}$ induces a $K_{4}$ in $G$. By the hypothesis of the theorem, there exists a vertex, says $v_{1}$, has degree at most four. If $\deg_{G}(v_{1}) = 3$, then $v_{1}$ is contained in a triangle $v_{1}v_{2}v_{3}$, which contradicts \autoref{No3InTriangle}. If $\deg_{G}(v_{1}) = 4$, then the edge $v_{1}v_{3}$ is contained in two adjacent triangles in $G$, which contradicts \autoref{No4InAdjacent}.
\end{proof}

\begin{claim}\label{OneBigVertex}%
Let $uvw$ be on the face boundary walk of a $4^{+}$-face of $G^{*}$. If $u$ is a true vertex of degree at most five and $uv$ is not a new edge, then at least one of $v$ and $w$ is a big vertex in $G^{*}$.
\end{claim}
\begin{proof}%
If $v$ is a true vertex, then $uv \in E(G)$ and $\deg_{G^{*}}(v) \geq \Delta - 2 \geq 9$ by \autoref{DEDGE}. So we may assume that $v$ is a crossing vertex and $w$ is a true vertex. By the construction of $G^{*}$, the vertex $w$ is a big vertex.
\end{proof}

\begin{figure}[!htb]%
\centering
\subcaptionbox{\label{fig:subfig:a}}{\includegraphics{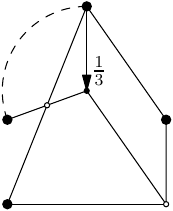}}\hfill~
\subcaptionbox{\label{fig:subfig:b}}{\includegraphics{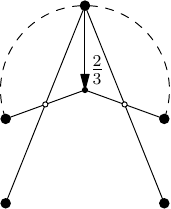}}\hfill~
\subcaptionbox{\label{fig:subfig:c}}{\includegraphics{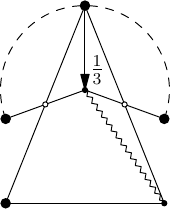}}\hfill~
\subcaptionbox{\label{fig:subfig:d}}{\includegraphics{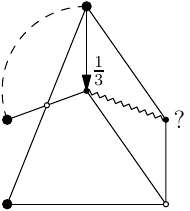}}
\caption{Discharging rules}
\label{fig:contfig:one}
\end{figure}

{\bf The Discharging Rules:}
\begin{enumerate}[label=(R\arabic*)]%
\item Every $\Delta$-vertex which is adjacent to some $(3, 3^{+})$-vertices of $G^{*}$ sends $1/2$ to a particular $\Delta$-vertex $v_{0}$, and every $(3, 3^{+})$-vertex of $G^{*}$ receives $1$ from the vertex $v_{0}$ (no matter whether these two vertices are adjacent).
\item Every $4^{+}$-face sends its redundant charge equally to its incident small vertices.
\item If $v$ is a $(5, 5)$-vertex and it is incident with five $3$-faces, then $v$ receives $1/3$ from each of its true neighbors.
\item All the other discharging rules are illustrated in figures (a)--(x); note that the dashed line denotes the two vertices are nonadjacent, the wavy line denotes the ``new edge'', and \includegraphics{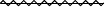} denotes the edge that we do not know whether it is a new edge; the solid dot labeled with "?" means that we cannot determine whether it is big or small vertex, the big solid dot denotes big true vertex, the small solid dot denotes small true vertex and the hollow dot denotes crossing vertex.
\end{enumerate}
\begin{claim}\label{BigFace}%
  Let $uvw$ be on the face boundary walk of a $4^{+}$-face $f$. Suppose that $v$ is a true vertex of degree at most five and neither $uv$ nor $vw$ is a new edge. If $f$ is a $4$-face, then $v$ receives at least $1$ from $f$, unless both $u$ and $w$ are crossing vertices and $v$ receives $2/3$ from $f$. If $f$ is a $5^{+}$-face, then $v$ receives at least $4/3$ from it.
\end{claim}
\begin{proof}%
  Suppose that $f = uvww'$ is a $4$-face. If both $u$ and $w$ are crossing vertices, then $f$ sends its redundant charge equally to three small vertices by \autoref{OneBigVertex} and (R2), and then $v$ receives $2/3$ from $f$. If  at least one vertex in $\{u, w\}$, says $u$, is a true vertex, then it is big, and additionally at least one vertex in $\{w, w'\}$ is a big vertex by \autoref{OneBigVertex}, and thus the vertex $v$ receives at least $1$ from $f$ by (R2).

  Suppose that $f$ is a $5^{+}$-face and $u'uvww'$ is on the face boundary walk of $f$. By \autoref{OneBigVertex}, at least one vertex in $\{u', u\}$ (similarly, at least one vertex in $\{w, w'\}$) is a big vertex, and then $f$ is incident with at least two big vertices. Hence, the vertex $v$ receives at least
\[
\frac{2\deg(f)-6}{\deg(f)-2} = 2 - \frac{2}{\deg(f)-2} \geq 2-2/3 = 4/3.\qedhere
\]
\end{proof}

From the discharging rules, we have the following claim.
\begin{claim}\label{Crossing0}%
  Let $w$ be a crossing vertex with a small neighbor $w_{1}$. If $w$ is incident with a $3$-face with face angle $w_{1}ww_{2}$ and $ww_{1}$ is incident with one $4^{+}$-face, then $w_{2}$ does not send charge to $w$.
\end{claim}

  From the discharging rules, the final charge of every face is nonnegative. So it suffices to consider the final charge of vertices in $G^{*}$. Let $v$ be an arbitrary vertex of $G^{*}$, we will analyze the vertex $v$ according to its degree.

  Suppose that $e_{0}, e_{1}, \dots, e_{t}$ are consecutive edges incident with a $(t + 1)^{+}$-vertex $x$ in counterclockwise order and $\deg(v) \geq \Delta - 2$, and the other end of $e_{i}$ is $x_{i}$ for $0 \leq i \leq t$. If both $x_{0}$ and $x_{t}$ receive $0$ from $x$ through $e_{0}$ and $e_{t}$ respectively, and $x_{i}$ receives positive charge from $x$ for $1 \leq i \leq t-1$, we call this local structure a {\em semi-fan with $t$ faces} and the vertex $x$ {\em center} of the semi-fan, call the edges $e_{i}$ {\em fan ribs}, and $e_{i-1}$ {\em precursor} of $e_{i}$ and $e_{i+1}$ {\em successor} of $e_{i}$. We show that the vertices receive charge from big vertices such that its final charge is nonnegative and in every semi-fan, the average charge sent out by the center is at most $2/5$, and then the final charge of every $(\Delta - 2)^{+}$-vertex is positive.

\begin{case}%
The vertex $v$ is a $(3, 3)$-vertex and $v_{1}, v_{2}, v_{3}$ are its neighbors.
\end{case}
If $v$ is incident with three $3$-faces, then \autoref{No4Clique} implies that one vertex in $\{v_{1}, v_{2}, v_{3}\}$ is a crossing vertex, thus the vertex $v$ must be contained in a triangle of $G$, which is a contradiction. Hence, the vertex $v$ is incident with at least one $4^{+}$-face.

\begin{subcase}%
Suppose that $v$ is incident with three $4^{+}$-faces. By \autoref{BigFace}, the vertex $v$ receives at least $2/3$ from each incident face, and then its final charge is at least $3 - 6 + 1 + 3 \times 2/3 = 0$ by (R1).
\end{subcase}

\begin{figure}[!htb]%
\ContinuedFloat
\centering
\subcaptionbox{\label{fig:subfig:e}}{\includegraphics{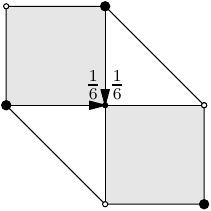}}\hfill~
\subcaptionbox{\label{fig:subfig:f}}{\includegraphics{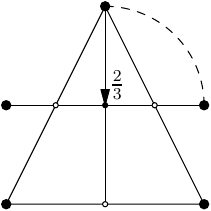}}\hfill~
\subcaptionbox{\label{fig:subfig:g}}{\includegraphics{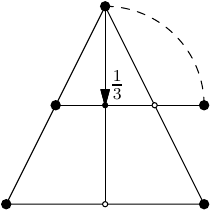}}\hfill~
\subcaptionbox{\label{fig:subfig:h}}{\includegraphics{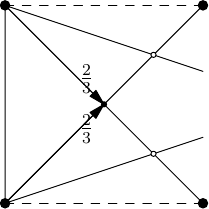}}\hfill~
\subcaptionbox{\label{fig:subfig:ii}}{\includegraphics{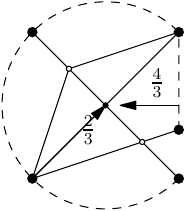}}\hfill~
\subcaptionbox{\label{fig:subfig:i}}{\includegraphics{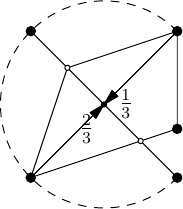}}\hfill~
\subcaptionbox{\label{fig:subfig:j}}{\includegraphics{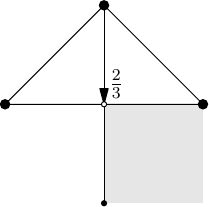}}\hfill~
\subcaptionbox{\label{fig:subfig:k}}{\includegraphics{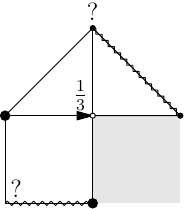}}\hfill~
\subcaptionbox{\label{fig:subfig:l}}{\includegraphics{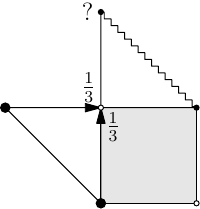}}
\label{fig:contfig:two}
\end{figure}

\begin{subcase}%
  Suppose that $v$ is incident with exactly two $4^{+}$-faces. By \autoref{BigFace}, if $v$ is incident with a $5^{+}$-face $f$, then it receives at least $4/3$ from $f$ and at least $2/3$ from the other $4^{+}$-face, and then its final charge is at least $3 - 6 + 1 + 4/3 + 2/3 = 0$. So we may assume that $v$ is incident with two $4$-faces and one $3$-face. We may assume that $v$ is incident with a $3$-face with face angle $v_{1}vv_{2}$ and $v_{1}$ is a true vertex. Thus, the vertex $v_{2}$ is a crossing vertex since $v$ is not contained in a triangle of $G$. By \autoref{BigFace}, if $v_{3}$ is a true vertex, then $v$ receives at least $1$ from each incident $4$-face, and its final charge is at least $3 - 6 + 1 + 2 \times 1 = 0$. So we may assume that $v_{3}$ is a crossing vertex, see Fig~\subref{fig:subfig:a}. By (R1) and \autoref{BigFace}, the final charge of $v$ is at least $3 - 6 + 1 + 1 + 2/3 + 1/3 = 0$. By \autoref{Crossing0}, we know that the vertex $v_{1}$ does not send  charge to $v_{2}$.
\end{subcase}

\begin{subcase}\label{subcase1-3}%
  Suppose that $v$ is incident with exactly one $4^{+}$-face $f$. By symmetry, assume that $f$ has a face angle $v_{2}vv_{3}$. If $v_{1}$ is a crossing vertex, then both $v_{2}$ and $v_{3}$ are true vertices, thus $v$ is contained in a triangle of $G$ induced by $\{v_{2}, v, v_{3}\}$, which contradicts \autoref{No3InTriangle}. So we may assume that $v_{1}$ is a true vertex but both $v_{2}$ and $v_{3}$ are crossing vertices for the same reason. Moreover, the $4^{+}$-face $f$ is a $5^{+}$-face; otherwise, there exist two multiple edges incident with $v_{1}$ in $G$, which is a contradiction. By (R1) and \autoref{BigFace}, the final charge of $v$ is at least $3 - 6 + 1 + 4/3 + 2/3 = 0$, see Fig~\subref{fig:subfig:b}. By \autoref{Crossing0}, neither $v_{2}$ nor $v_{3}$ receives charge from $v_{1}$.
\end{subcase}
  Let $v_{1}, v_{2}, \dots, v_{l}$ be consecutive neighbors of an $l$-vertex $v$ in counterclockwise order, and let $f_{i}$ be the incident face with face angle $v_{i}vv_{i-1}$, where the subtraction of subscript is taken modulo $l$.
\begin{case}%
  The vertex $v$ is a $(3, 4)$-vertex, that is, $v$ is a $4$-vertex in $G^{*}$ and it is incident with a new edge.
\end{case}
\begin{subcase}%
  If $v$ is incident with at least two $4^{+}$-faces, then the vertex $v$ receives at least 2/3 from each incident $4^{+}$-face, and thus its final charge is at least $4 - 6 + 1 + 2 \times 2/3 = 1/3$ by (R1).
\end{subcase}
\begin{subcase}%
  Suppose that $v$ is incident with exactly one $4^{+}$-face $f_{3}$. If $v$ receives at least $1$ from $f_{3}$, then its final charge is at least $4 - 6 + 1 + 1 = 0$. So we may assume that $v$ receives less than $1$ from $f_{3}$. By \autoref{OneBigVertex} and (R2), the face $f_{3}$ is a $4$-face incident with exactly one big vertex and it sends $2/3$ to $v$.

(1) Suppose that the new edge incident with $v$ is incident with $f_{3}$. By symmetry, we may assume that $vv_{3}$ is the new edge. By similar arguments as in subcase \ref{subcase1-3}, we may assume that $v_{1}$ is a true vertex and both $v_{2}$ and $v_{4}$ are crossing vertices, see Fig~\subref{fig:subfig:c}. By the discharging rules, the final charge of $v$ is $4 - 6 + 1+ 2/3 + 1/3 = 0$. By \autoref{Crossing0}, the vertex $v_{1}$ does not send charge to $v_{2}$.

(2) The new edge incident with $v$ is not incident with $f_{3}$. By \autoref{BigFace}, both $v_{2}$ and $v_{3}$ are crossing vertices, thus both $v_{1}$ and $v_{4}$ are true vertices, see Fig~\subref{fig:subfig:d}. By the discharging rules, the final charge of $v$ is $4 - 6 + 1+ 2/3 + 1/3 = 0$. By \autoref{Crossing0}, we also know that the vertex $v_{1}$ does not send charge to $v_{2}$.
\end{subcase}
\begin{subcase}%
  Suppose that $v$ is incident with four $3$-faces and $vv_{1}$ is the new edge incident with $v$. If $v_{3}$ is a crossing vertex, then both $v_{2}$ and $v_{4}$ are true vertices, and then $v$ is contained in a triangle induced by $\{v_{2}, v, v_{4}\}$ in $G$, which contradicts \autoref{No3InTriangle}. So we may assume that $v_{3}$ is a true vertex, thus \autoref{No3InTriangle} implies that both $v_{2}$ and $v_{4}$ are crossing vertices, but there are two multiple edges in $G$ with ends $v_{1}$ and $v_{3}$, which is a contradiction. Therefore, it is impossible to have four $3$-faces incident with $v$.
\end{subcase}

\begin{case}%
The vertex $v$ is a $(4, 4)$-vertex.
\end{case}

\begin{subcase}%
  Suppose that $v$ is incident with at least three $4^{+}$-faces. By \autoref{BigFace}, the vertex $v$ receives at least $2/3$ from each incident $4^{+}$-face, then its final charge is at least $4 - 6 + 3 \times 2/3 = 0$.
\end{subcase}
\begin{subcase}%
  Suppose that $v$ is incident with exactly two $4^{+}$-faces. If $v$ receives at least $1$ from each incident $4^{+}$-face, then its final charge is at least $4 - 6  + 2 \times 1 = 0$. By \autoref{BigFace}, we may assume that $v$ receives $2/3$ from one of its incident $4$-face with face angle $v_{3}vv_{4}$. Moreover, both $v_{3}$ and $v_{4}$ are crossing vertices. If $v$ receives at least $4/3$ from the other $4^{+}$-face, then its final charge is at least $4 - 6 + 2/3 + 4/3 = 0$. So we may assume that the other $4^{+}$-face $f$ sends at most $1$ to $v$. By \autoref{BigFace}, the face $f$ is also a $4$-face.

  Further, suppose that the two $4$-faces are nonadjacent. Thus, both $v_{1}$ and $v_{2}$ are true vertices. By (R2), the vertex $v$ receives at least $1$ from $f_{2}$; in fact, the vertex $v$ receives $1$ from $f_{2}$, it follows that $f_{2}$ is incident with one crossing vertex, see Fig~\subref{fig:subfig:e}. Hence, the final charge of $v$ is at least $4 - 6 + 1 + 2/3 + 2 \times 1/6 = 0$. By \autoref{Crossing0}, the vertex $v_{1}$ does not send charge to $v_{4}$; similarly, the vertex $v_{2}$ does not send charge to $v_{3}$.

  So we may assume that the two $4$-faces are adjacent. By symmetry, we may assume that $vv_{3}$ is incident with two $4$-faces. If $v_{2}$ is a crossing vertex, then the final charge of $v$ is $4 - 6 + 3 \times 2/3 = 0$, see Fig~\subref{fig:subfig:f}; if $v_{2}$ is a true vertex, then the final charge of $v$ is $4 - 6 + 1 + 2/3 + 1/3 = 0$, see Fig~\subref{fig:subfig:g}. By \autoref{Crossing0}, neither $v_{2}$ nor $v_{4}$ receives charge from $v_{1}$.
\end{subcase}
\begin{subcase}%
  Suppose that $v$ is incident with exactly one $4^{+}$-face having a face angle $v_{1}vv_{4}$. Firstly, assume that both $v_{2}$ and $v_{3}$ are true vertices. Thus both $v_{1}$ and $v_{4}$ are crossing vertices since $v$ is not contained in two adjacent triangles in $G$, see Fig~\subref{fig:subfig:h}. By the discharging rules, the final charge of $v$ is at least $4 - 6 + 3 \times 2/3 = 0$. Moreover, by \autoref{Crossing0}, we know that $v_{2}$ does not send charge to $v_{1}$; similarly,  the vertex $v_{3}$ does not send charge to $v_{4}$. In a semi-fan, if $v_{2}v$ or $v_{3}v$ is a fan rib, then the average charge sent out by the center vertex is $1/3$.

  Secondly, assume that one of $v_{2}$ and $v_{3}$, says $v_{2}$, is a crossing vertex. The vertices $v_{1}$ and $v_{3}$ are all true vertices since crossing vertices are independent. \autoref{No4InAdjacent} implies that $v_{4}$ is a crossing vertex. By \autoref{No4InAdjacent}, the crossing vertex $v_{2}$ is incident with two $4^{+}$-faces, and $v_{2}$ receives $0$ from its neighbors. By \autoref{Crossing0}, the crossing vertex $v_{4}$ also receives $0$ from $v_{3}$. In a semi-fan, if $v_{3}v$ is a fan rib, then the average charge sent out by the center $v_{3}$ is $1/3$.

(i) If $f_{1}$ is a $5^{+}$-face, then $v$ receives at least $4/3$ from $f_{1}$ by \autoref{BigFace}, and receives $2/3$ from $v_{3}$, and thus the final charge of $v$ is at least $4 - 6 + 4/3 + 2/3 = 0$, see Fig~\subref{fig:subfig:ii}.

(ii) If $v$ is incident with a $4$-face $f_{1} = v_{1}vv_{4}v^{*}$, then $v$ receives $1, 2/3$ and $1/3$ from $f_{1}, v_{3}$ and $v_{1}$ respectively, and then the final charge of $v$ is $4 - 6 + 1 + 2/3 + 1/3 = 0$, see Fig~\subref{fig:subfig:i}. Note that the vertex $v^{*}$ is a big vertex and $v_{1}$ sends $0$ to $v^{*}$. The crossing vertex $v_{2}$ receives at least $1$ from each incident $4^{+}$-face such that its final charge is nonnegative, thus neither $v_{1}$ nor $v_{3}$ sends charge to $v_{2}$. By \autoref{Crossing0}, the vertex $v_{3}$ does not send charge to $v_{4}$. So if $v_{3}$ or $v_{1}v$ is in a semi-fan, then the average charge sent out by the center is at most $1/3$. 
\end{subcase}
\begin{subcase}
  Suppose that $v$ is incident with four $3$-faces. By \autoref{No4InAdjacent}, the vertex $v$ is not contained in adjacent triangles of $G$, then $v$ is incident with at least two crossing vertices. Consequently, we may assume that $v_{2}$ and $v_{4}$ are crossing vertices since the crossing vertices are independent, and then both $v_{1}$ and $v_{3}$ are true vertices, but there are two multiple edges of $G$ with ends $v_{1}$ and $v_{3}$, a contradiction. Therefore, it is impossible to have four $3$-faces incident with $v$.
\end{subcase}

\begin{case}%
The vertex $v$ is a crossing vertex.
\end{case}
Clearly, all the neighbors of $v$ are true vertices.
\begin{subcase}%
  Suppose that $v$ is incident with at least three $4^{+}$-faces. Note that every $4^{+}$-face is incident with at least one big vertex, it follows that $v$ receives at least $2/3$ from each incident $4^{+}$-face. Therefore, the final charge of $v$ is at least $4 - 6 + 3 \times 2/3 = 0$.
\end{subcase}

\begin{subcase}%
  Suppose that $v$ is incident with exactly two $4^{+}$-faces. If $v$ receives at least $1$ from each incident $4^{+}$-face, then its final charge is at least $4 - 6 + 2 \times 1 = 0$. So we may assume that $v$ receives $2/3$ from its incident $4^{+}$-face $f_{4}$. Thus, the face $f_{4}$ is a $4$-face which is incident with exactly one big vertex. In particular, exactly one of $v_{3}$ and $v_{4}$ is a big vertex. If $v$ receives at least $4/3$ from the other $4^{+}$-face, then its final charge is at least $4 - 6 + 2/3 + 4/3 = 0$. So we may assume that $v$  receives less than $4/3$ from the other $4^{+}$-face.

(1) Suppose that the two $3$-faces are adjacent, says $f_{1}$ and $f_{2}$. First of all, we assume that $v_{3}$ is a small vertex. From the construction of $G^{*}$, we know that both $v_{2}$ and $v_{4}$ are big vertices. By \autoref{DEDGE}, we have $\deg_{G}(v_{1}) \geq \Delta - 2$. By the discharging rules, the vertex $v$ receives at least $2/3$ from each incident $4^{+}$-face and $2/3$ from $v_{1}$, thus the final charge of $v$ is at least $4 - 6 + 3 \times 2/3 = 0$, see Fig~\subref{fig:subfig:j}. In a semi-fan, if $v_{1}v$ is a fan rib, then the average charge sent out by the center $v_{1}$ is $1/3$.

  Next, we may assume that $v_{3}$ is a big vertex in $G^{*}$. Note that $v_{4}$ is a small vertex, and thus $v_{2}$ is a vertex of degree at least $\Delta -2$. Note that $f_{3}$ sends less than $4/3$ to $v$ and it is incident with at least two big vertices, thus $f_{3}$ is a $4$-face incident with exactly two big vertices and it sends $1$ to $v$, see Fig~\subref{fig:subfig:k}. Therefore, the final charge of $v$ is $4 - 6 + 1 + 2/3 + 1/3 = 0$.

(2) Suppose that the two $3$-faces are nonadjacent. It follows that $f_{1}$ and $f_{3}$ are the two $3$-faces. Suppose that one of $v_{1}v_{4}$ and $v_{2}v_{3}$ is a new edge. By symmetry, we may assume that $v_{1}v_{4}$ is a new edge, thus both $v_{2}$ and $v_{3}$ have degree at least $\Delta - 2$. By the discharging rules, the vertex $v$ receives at least $2/3$ from $f_{2}$ and $1/3$ from each of $v_{2}$ and $v_{3}$, thus its final charge is at least $4 - 6 + 2 \times 2/3 + 2 \times 1/3 = 0$, see Fig~\subref{fig:subfig:l}. Notice that $f_{4}$ is incident with only one big vertex $v_{3}$, thus $v^{*}$ is a crossing vertex and $v_{4}$ is a $(3, 4)$- or $(4, 5)$-vertex.

  By symmetry, we may assume that both $v_{1}v_{4}$ and $v_{2}v_{3}$ are edges of $G$. By the symmetry of $v_{3}$ and $v_{4}$, assume that $v_{4}$ is a small vertex. Hence, both $v_{1}$ and $v_{2}$ have degree at least $\Delta - 2$ and $v_{3}$ is a big vertex. By the discharging rules, the vertex $v$ receives at least $1$ from $f_{2}$ and $1/3$ from $v_{2}$, then its final charge is at least $4 - 6 + 1 + 2/3 + 1/3 = 0$, see Fig~\subref{fig:subfig:m}. Note that $f_{2}$ is a $4$-face with exactly two big vertices. In a semi-fan, if the center sends $1/3$ to such a crossing vertex $v$, then it sends out $0$ through its precursor or successor at it.
\end{subcase}
\begin{figure}%
\ContinuedFloat
\centering
\subcaptionbox{\label{fig:subfig:m}}{\includegraphics{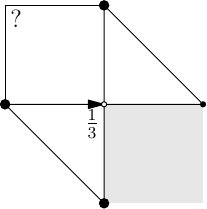}}\hfill~
\subcaptionbox{\label{fig:subfig:n}}{\includegraphics{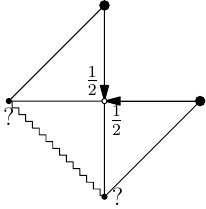}}\hfill~
\subcaptionbox{\label{fig:subfig:o}}{\includegraphics{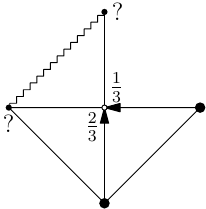}}\hfill~
\subcaptionbox{\label{fig:subfig:p}}{\includegraphics{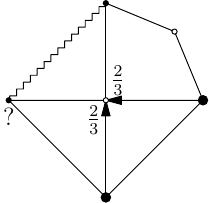}}
\end{figure}

\begin{subcase}%
Suppose that $v$ is incident with exactly one $4^{+}$-face. Without loss of generality, assume that $f_{1}$ is the $4^{+}$-face.

(1) Suppose that the edge $v_{2}v_{3}$ is a new edge. By \autoref{DEDGE}, both $v_{1}$ and $v_{4}$ have degree at least $\Delta - 2$. By the discharging rules, the vertex $v$ receives at least $1$ from $f_{1}$ and $1/2$ from each of $v_{1}$ and $v_{4}$, thus the final charge of $v$ is at least $4 - 6 + 1 + 2 \times 1/2 = 0$, see Fig~\subref{fig:subfig:n}.

(2) Suppose that one of $v_{1}v_{2}$ and $v_{3}v_{4}$ is a new edge. By symmetry, assume that $v_{1}v_{2}$ is a new edge. By \autoref{DEDGE}, both $v_{3}$ and $v_{4}$ have degree at least $\Delta -2$. If $v$ receives at least $1$ from $f_{1}$, then $v$ receives $2/3$ from $v_{3}$ and $1/3$ from $v_{4}$, and then its final charge is at least $4 - 6 + 1 + 2/3 + 1/3 = 0$, see Fig~\subref{fig:subfig:o}. So we may assume that $v$ receives less than $1$ from $f_{1}$. In fact, $f_{1}$ is a $4$-face incident with only one big vertex $v_{4}$. So we may assume that $f_{1} = v_{1}vv_{4}v^{*}$ and $v^{*}$ is a crossing vertex and $v_{1}$ is a $(3, 4)$- or $(4, 5)$-vertex, see Fig~\subref{fig:subfig:p}. Now, the vertex $v$ receives $2/3$ from $f_{1}$ and $2/3$ from each of $v_{3}$ and $v_{4}$, and then its final charge is $4 - 6 + 3 \times 2/3 = 0$. By the discharging rules, we can check that $v_{4}$ does not send charge to $v^{*}$. Hence, if $v_{4}v$ is a fan rib, then the average charge sent out by the center $v_{4}$ is $1/3$.

(3) The edges $v_{1}v_{2}, v_{2}v_{3}$ and $v_{3}v_{4}$ are all edges of $G$. By \autoref{No4Clique}, the vertices $v_{1}$ and $v_{4}$ are not adjacent in $G$, thus $\{v_{1}, v_{2}, v_{3}, v_{4}\}$ induces a $K_{4}^{-}$ in $G$. By \autoref{No3InTriangle} and the condition $\mathcal{P}$, one vertex in $\{v_{2}, v_{3}\}$, says $v_{2}$, is a vertex of degree at most five in $G$. Hence, each vertex in $\{v_{1}, v_{3}, v_{4}\}$ has degree at least $\Delta -2$. The vertex $v$ receives at least $1$ from $f_{1}$ and $1/3$ from each of $v_{1}, v_{3}$ and $v_{4}$, see Fig~\subref{fig:subfig:q}. Therefore, the final charge of $v$ is at least $4 - 6 + 1 + 3 \times 1/3 = 0$.
\end{subcase}
\begin{subcase}%
  Suppose that $v$ is incident with four $3$-faces. By \autoref{No4Clique}, one of the four $3$-faces is incident with a new edge. Suppose that $v_{1}v_{2}$ is a new edge. By \autoref{DEDGE}, both $v_{3}$ and $v_{4}$ have degree at least $\Delta - 2$, and $\{v_{1}, v_{2}, v_{3}, v_{4}\}$ induces a $K_{4}^{-}$ in $G$. By the condition~(2) in $\mathcal{P}$, we have that $\min\{\deg_{G}(v_{1}), \deg_{G}(v_{2})\} = 3$, which contradicts \autoref{No3InTriangle}. 
\end{subcase}
\begin{case}%
The vertex $v$ is a $(4, 5)$-vertex or $(5, 5)$-vertex.
\end{case}
\begin{subcase}%
  If $v$ is incident with at least two $4^{+}$-faces, then $v$ receives at least $2/3$ from each incident $4^{+}$-face, thus the final charge is at least $5 - 6 + 2 \times 2/3 = 1/3$.
\end{subcase}
\begin{figure}%
\ContinuedFloat
\centering
\subcaptionbox{\label{fig:subfig:q}}{\includegraphics{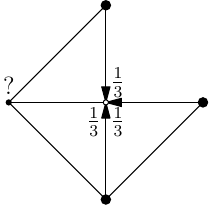}}\hfill~
\subcaptionbox{\label{fig:subfig:r}}{\includegraphics{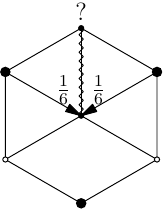}}\hfill~
\subcaptionbox{\label{fig:subfig:s}}{\includegraphics{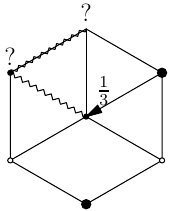}}\hfill~
\subcaptionbox{\label{fig:subfig:t}}{\includegraphics{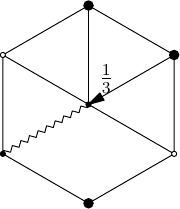}}\hfill~
\subcaptionbox{\label{fig:subfig:u}}{\includegraphics{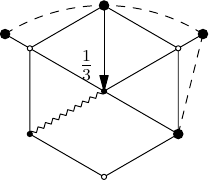}}\hfill~
\subcaptionbox{\label{fig:subfig:v}}{\includegraphics{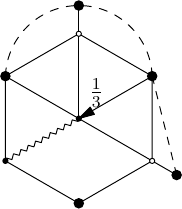}}\hfill~
\subcaptionbox{\label{fig:subfig:w}}{\includegraphics{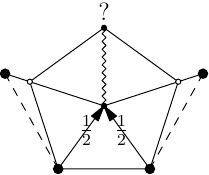}}\hfill~
\subcaptionbox{\label{fig:subfig:x}}{\includegraphics{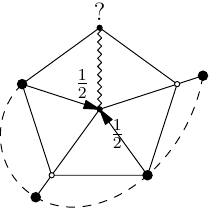}}
\end{figure}
\begin{subcase}%
  Suppose that $v$ is incident with exactly one $4^{+}$-face, says $f_{4}$. If $v$ receives at least $1$ from $f_{4}$, then the final charge is at least $5 - 6 + 1 = 0$. So we may assume that $v$ receives less than $1$ from $f_{4}$. By \autoref{OneBigVertex} and (R2), the face $f_{4}$ is a $4$-face incident with exactly one big vertex. Clearly, the vertex $v$ receives $2/3$ from $f_{4}$. 

  Firstly, assume that both $v_{3}$ and $v_{4}$ are crossing vertices, thus both $v_{2}$ and $v_{5}$ are true vertices. If neither $vv_{2}$ nor $vv_{5}$ is a new edge, then $v$ receives $1/6$ from each of $v_{2}$ and $v_{5}$, thus its final charge is $5 - 6 + 2/3 + 2 \times 1/6 = 0$, see Fig~\subref{fig:subfig:r}. By symmetry, if one of $vv_{2}$ and $vv_{5}$, says $vv_{2}$, is a new edge, then $v$ receives $1/3$ from $v_{5}$, and the final charge is $5 - 6 + 2/3 + 1/3 =0$, see Fig~\subref{fig:subfig:s}; note that we do not know whether $v_{1}$ is a true vertex. By the discharging rules, the vertex $v_{5}$ does not send charge to $v_{4}$.

  Secondly, assume that one of $v_{3}$ and $v_{4}$ is a true vertex. By \autoref{BigFace}, one of $vv_{3}$ and $vv_{4}$ is a new edge. By symmetry, we may assume that $vv_{3}$ is a new edge. If both $v_{1}$ and $v_{5}$ are true vertices, then the $4$-vertex $v$ in $G$ is contained in a triangle $v_{1}vv_{5}$, and \autoref{No4InAdjacent} implies that $v_{2}$ and $v_{4}$ are crossing vertices, thus the local structure is as illustrated in Fig~\subref{fig:subfig:t}; if $v_{1}$ is a true vertex and $v_{5}$ is a crossing vertex, then the local structure is as illustrated in Fig~\subref{fig:subfig:u}; if $v_{1}$ is a crossing vertex and $v_{5}$ is a true vertex, then the local structure is as illustrated in Fig~\subref{fig:subfig:v}. Anyway, the vertex $v$ receives $1/3$ from its neighbors in $G^{*}$, and the final charge of $v$ is $5 - 6 + 2/3 + 1/3 = 0$.
\end{subcase}

\begin{subcase}%
The vertex $v$ is incident with five $3$-faces.

(1) If $v$ is a $(5, 5)$-vertex, then at least three of its neighbors in $G^{*}$ are true vertices and have degree at least $\Delta - 2$, and then the final charge of $v$ is at least $5 - 6 + 3 \times 1/3 = 0$ by (R3).

(2) Suppose that $v$ is a $(4, 5)$-vertex and $vv_{1}$ is a new edge.

  If both $v_{3}$ and $v_{4}$ are true vertices, then both $v_{2}$ and $v_{5}$ are crossing vertices and the local structure is illustrated in Fig~\subref{fig:subfig:w}, the vertex $v$ receives $1/2$ from each of $v_{3}$ and $v_{4}$, and its final charge is $5 - 6 + 2 \times 1/2 = 0$.

  If one of $v_{3}$ and $v_{4}$, says $v_{3}$, is a crossing vertex, then the local structure is illustrated in Fig~\subref{fig:subfig:x}. The vertex $v$ receives $1/2$ from each of $v_{2}$ and $v_{4}$, thus the final charge is $5 - 6 + 2 \times 1/2 = 0$.
\end{subcase}

  From now on, we will check the final charge of every big vertex. If $v$ is a $(3, 5)$-vertex, then the final charge is $5-6+1=0$ by (R1). If $v$ has degree six, then the final charge is zero. If the degree of $v$ in $G^{*}$ belongs to $\{7, 8, \dots, \Delta - 3\}$, then its final charge is positive. The remaining case is that of degree at least $\Delta -2$. Note that $\Delta - 2 \geq 9$.

  Suppose that the center vertex sends out a $2/3$ through a fan rib in a semi-fan. If the center vertex sends a $2/3$ to a small vertex as illustrated in Fig~\subref{fig:subfig:b} or \subref{fig:subfig:f} or \subref{fig:subfig:h} or \subref{fig:subfig:i} or \subref{fig:subfig:j}, then we have proved that the average charge sent out by the center is $1/3$.

  Suppose that the center is the vertex $v_{3}$ in Fig~\subref{fig:subfig:o} or \subref{fig:subfig:p}. If $v_{3}$ does not send charge to $v_{2}$, then the average charge sent out by $v_{3}$ is $1/3$. So we may assume that $v_{3}$ sends a positive charge to the small vertex $v_{2}$. Thus, the vertex $v_{2}$ must be a $(3, 4)$- or $(4, 5)$-vertex; note that the small vertex $v_{2}$ is incident with a new edge. Note that $v_{2}$ is contained in a triangle $v_{2}v_{3}v_{4}$ of $G$, thus \autoref{No3InTriangle} implies that $v_{2}$ is a $(4, 5)$-vertex. By the discharging rules, the vertex $v_{3}$ sends $1/3$ to this $(4, 5)$-vertex (c.f. Fig~\subref{fig:subfig:s}), but $v_{3}$ sends out $0$ through its precursor or successor, hence the average charge sent out by the center is $(2/3+1/3)/3 = 1/3$.

  If the center is the vertex $v_{4}$ as illustrated in Fig~\subref{fig:subfig:p} and it sends out a $2/3$ to the crossing vertex, then we have proved that the average charge sent out by it is $1/3$.

  In what follows, we assume that the center does not send a $2/3$ through fan ribs.

  Suppose that the center sends out a $1/2$ through a fan rib. Assume that the center sends a $1/2$ to a crossing vertex. By symmetry, we may assume that the center is the vertex $v_{1}$ as illustrated in Fig~\subref{fig:subfig:n}. If $v_{1}$ sends a positive charge to $v_{2}$, then $v_{2}$ must be a small vertex, that is, it is a $(3, 4)$- or $(4, 5)$-vertex. If $v_{2}$ is a $(3, 4)$-vertex, then $v_{1}$ sends $1/3$ to $v_{2}$ and it sends $0$ through the precursor or successor (c.f. Fig~\subref{fig:subfig:c}). If $v_{2}$ is a $(4, 5)$-vertex and it is the vertex $v$ as illustrated in Fig~\subref{fig:subfig:s}, then the center sends $1/3$ to such a $(4, 5)$-vertex and sends $0$ through the precursor or successor. If $v_{2}$ is a $(4, 5)$-vertex and it is the vertex $v$ as illustrated in Fig~\subref{fig:subfig:u}, then the center sends $1/3$ to such a $(4, 5)$-vertex and sends $0$ through the precursor or successor. If $v_{2}$ is a $(4, 5)$-vertex and it is the vertex $v$ as illustrated in Fig~\subref{fig:subfig:w} or Fig~\subref{fig:subfig:x}, then the center sends $1/2$ to such a $(4, 5)$-vertex, but it sends $0$ to through the precursor or successor. By the above arguments, if $v$ sends out a $1/2$ through a fan rib, then it sends at most two $1/2$ between this fan rib and the big rib. Therefore, in a semi-fan, the center sends out at most four $1/2$, and then the average charge sent out by the center is at most $(4 \times 1/2)/5 = 2/5$, the equality holds if and only if the semi-fan contains five faces and the center sends out four $1/2$.

  Suppose that the center sends out a $1/2$ to a $(4, 5)$-vertex, but does not send charge to crossing vertices. By the discharging rules, the center sends out at most two $1/2$ in a semi-fan, and then the average charge sent out by the center is at most $(2 \times 1/2 + (k-3)\times 1/3)/k = 1/3$.

  If the center sends out at most $1/3$ through each fan ribs, then the average charge sent out by the center is less than $1/3$.

  If $v$ is a $(\Delta-2)$-vertex, then it only sends positive charge to crossing vertices or $(5, 5)$-vertices, then the average charge sent out by the center is at most $1/3$, the final charge of $v$ is at least $\Delta -2  - 6 - (\Delta -2)\times 1/3 \geq 0$; the equality holds if and only if $\Delta = 11$ and the average charge sent out by the center $v$ in every semi-fan is exactly $1/3$.

  If $v$ is a $(\Delta -1)$-vertex, then it only sends positive charge to crossing vertices or $(4, 4)$- or $(4, 5)$- or $(5, 5)$-vertices, thus the average charge sent out by the center is at most $2/5$, the final charge of $v$ is at least $(\Delta -1) - 6 - (\Delta - 1) \times 2/5 \geq 0$; the equality holds if and only if $\Delta = 11$ and the average charge sent out by the center in every semi-fan is $2/5$.

  If $v$ is a $\Delta$-vertex and it is not adjacent to $3$-vertex of $G$, then the final charge of $v$ is at least $\Delta - 6 - \Delta \times 2/5 >0$; if $v$ is a $\Delta$-vertex and $v$ is adjacent to some $3$-vertices of $G$, then its final charge is at least $\Delta - 6 - 1/2 - \Delta \times 2/5 > 0$.

\begin{claim}
Every vertex with maximum degree has positive final charge.
\end{claim}
\begin{proof}
  By \autoref{DEDGE}, there exists at least one vertex having degree at least $7$. Let $w$ be an arbitrary vertex of $G^{*}$ with maximum degree. If $\deg_{G^{*}}(w) \in \{7, \dots, \Delta -3\}$, then the final charge of $w$ is $\deg_{G^{*}}(w) - 6 > 0$. If $\deg_{G^{*}}(w) = \Delta$, then the final charge of $w$ is positive.

  If $\deg_{G^{*}}(w)= \Delta - 1$, then it cannot send charge to the $(4, 5)$-vertex as illustrated in Fig~\subref{fig:subfig:w} or ~\subref{fig:subfig:x} by \autoref{SumDelta3}. Therefore, the average charge sent out by $w$ is less than $2/5$ and the final charge of $w$ is positive.

  If $\deg_{G^{*}}(w) = \Delta - 2$, then it cannot send charge to the $(5, 5)$-vertex by \autoref{SumDelta3}. Therefore, the vertex $w$ can only send charge to crossing vertices, and the average charge sent out by the center $w$ is less than $1/3$, thus the final charge of $w$ is positive.

  Hence, the final charge of vertices with maximum degree is positive.
\end{proof}

Therefore, the sum of the final charge of each element is positive, which derives a contradiction. This complete the proof of theorem.
\begin{remark}%
  Zhang \etal \cite{MR3352881} proved that TCC holds for $1$-planar graphs with maximum degree at least~$13$. We can also  extend this result to $1$-toroidal graphs with maximum degree at least~$13$ by using similar techniques in this paper.
\end{remark}

\begin{problem}%
Does this method can be used to prove similar result for the diamond-free $1$-toroidal graphs?
\end{problem}

\vskip 3mm \vspace{0.3cm} \noindent{\bf Acknowledgments.} This project was supported by the National Natural Science Foundation of China (11101125) and partially supported by the Fundamental Research Funds for Universities in Henan. The author would like to thank the anonymous reviewers for their valuable comments and assistance on earlier drafts.


\begin{thebibliography}{10}

\bibitem{Behzad1965}
M.~Behzad, Graphs and their chromatic numbers, Ph.D. thesis, Michigan State
  University (1965).

\bibitem{MR832128}
O.~V. Borodin, Solution of the {R}ingel problem on vertex-face coloring of
  planar graphs and coloring of {$1$}-planar graphs, Metody Diskret. Analiz.
  (1984)~(41) 12--26, 108.

\bibitem{MR977440}
O.~V. Borodin, On the total coloring of planar graphs, J. Reine Angew. Math.
  394 (1989) 180--185.

\bibitem{MR1333779}
O.~V. Borodin, A new proof of the {$6$} color theorem, J. Graph Theory 19
  (1995)~(4) 507--521.

\bibitem{MR1865580}
O.~V. Borodin, A.~V. Kostochka, A.~Raspaud and E.~Sopena, Acyclic colouring of
  1-planar graphs, Discrete Appl. Math. 114 (2001)~(1-3) 29--41.

\bibitem{MR1304254}
T.~R. Jensen and B.~Toft, Graph coloring problems, Wiley-Interscience Series in
  Discrete Mathematics and Optimization, John Wiley \& Sons, Inc., New York,
  1995.

\bibitem{MR0453576}
A.~V. Kostochka, The total coloring of a multigraph with maximal degree {$4$},
  Discrete Math. 17 (1977)~(2) 161--163.

\bibitem{MR1425788}
A.~V. Kostochka, The total chromatic number of any multigraph with maximum
  degree five is at most seven, Discrete Math. 162 (1996)~(1-3) 199--214.

\bibitem{MR0187232}
G.~Ringel, {Ein Sechsfarbenproblem auf der Kugel}, Abh. Math. Sem. Univ.
  Hamburg 29 (1965)~(1) 107--117.

\bibitem{MR0278995}
M.~Rosenfeld, On the total coloring of certain graphs, Israel J. Math. 9
  (1971)~(3) 396--402.

\bibitem{MR1684286}
D.~P. Sanders and Y.~Zhao, On total 9-coloring planar graphs of maximum degree
  seven, J. Graph Theory 31 (1999)~(1) 67--73.

\bibitem{MR0285447}
N.~Vijayaditya, On total chromatic number of a graph, J. London Math. Soc. (2)
  3 (1971)~(3) 405--408.

\bibitem{Vizing1965}
V.~G. Vizing, Critical graphs with given chromatic class, Metody Diskret.
  Analiz. 5 (1965) 9--17.

\bibitem{MR0976059}
H.~P. Yap, Total colourings of graphs, Bull. London Math. Soc. 21 (1989)~(2)
  159--163.

\bibitem{MR3352881}
X.~Zhang, J.~Hou and G.~Liu, On total colorings of 1-planar graphs, J. Comb.
  Optim. 30 (2015)~(1) 160--173.

\bibitem{MR2876230}
X.~Zhang and G.~Liu, On edge colorings of 1-planar graphs without adjacent
  triangles, Inform. Process. Lett. 112 (2012)~(4) 138--142.

\bibitem{MR2965951}
X.~Zhang, J.~Wu and G.~Liu, List edge and list total coloring of 1-planar
  graphs, Front. Math. China 7 (2012)~(5) 1005--1018.

\bibitem{MR2779909}
X.~Zhang and J.-L. Wu, On edge colorings of 1-planar graphs, Inform. Process.
  Lett. 111 (2011)~(3) 124--128.

\end{thebibliography}
\end{document}